\documentclass[9pt]{IEEEtran}
\IEEEoverridecommandlockouts
\usepackage{isolatin1}
\usepackage{times}
\usepackage{pdfsync}
\usepackage{epsfig}
\usepackage{epstopdf}
\usepackage{amsmath}
\usepackage{amssymb}
\usepackage{mathrsfs}
\usepackage{color}
\usepackage{cite}
\usepackage{fix2col}

\newcommand{\eq}{\triangleq}

\newcommand{\field}[1]{\mathbb{#1}}
\newcommand{\R}{\field{R}}
\newcommand{\N}{\field{N}}
\newcommand{\U}{\field{U}}

\newcommand{\K}{\field{K}}
\newcommand{\Prob}{\mathbf{Pr}}
\newcommand{\E}{\mathbf{E}}
\newcommand{\hfs}{\hfill\ensuremath{\square}}

\newtheorem{thm}{Theorem}
\newtheorem{defi}{Definition}

\newtheorem{lem}{Lemma}
\newtheorem{rem}{Remark}
\newtheorem{ex}{Example}
\newtheorem{ass}{Assumption}

\graphicspath{{Matlab/}}

\title{Stability of Sequence-Based Control with Random Delays and  Dropouts}
 \author{Daniel~E.~Quevedo,~\IEEEmembership{Member,~IEEE,} and 
   Isabel Jurado
  \thanks{Daniel Quevedo is  with the School of
    Electrical Engineering \& 
      Computer Science, The University of Newcastle, NSW
      2308, Australia; e-mail:
      dquevedo@ieee.org. Isabel Jurado is with Departamento de Ingenier\'{\i}a de
      Sistemas y Autom\'atica, Escuela Superior de Ingenieros, Universidad de
      Sevilla, Spain; e-mail: ijurado@cartuja.us.es. This research was supported
      under Australian Research Council's 
  Discovery Projects funding scheme (project number DP0988601).} 
}

\begin{document}
\maketitle

\begin{abstract}
We study networked control of non-linear systems where system states and
tentative plant input sequences  are transmitted over  unreliable communication
channels. The  sequences are calculated recursively by using a pre-designed
nominally stabilizing state-feedback control mapping to plant state
predictions. The controller does not require receipt  
acknowledgments   or   knowledge of delay or dropout
distributions. For the  i.i.d.\  case, in which case the numbers of
consecutive dropouts are geometrically distributed, we show how the resulting closed loop system can be
modeled as a Markov non-linear jump system and establish sufficient conditions
for stochastic stability.   
\end{abstract}

\section{Introduction}
\label{sec:introduction}
There exists significant interest in implementing closed loop
control systems using   wireless  
communication networks. Such \emph{Networked Control Systems} (NCSs) are easier
to install and to maintain than their traditional counterparts equipped with dedicated
hardwired links. However, using off-the-shelf network technology for 
closed-loop control gives rise to various challenges. In particular, 
%links 
%between plant(s) and controller(s) 
%are not transparent and, thus, need to be taken into account in the design. In
%addition 
%to being  bit-rate 
%limited\cite{naifag07,quesil09}, 
practical communication channels are commonly
affected by packet dropouts and time-delays  (for example, due to fading and congestion)  %, mainly due 
%to transmission errors and waiting times to access the medium
\cite{chejoh11}. %,hesnag07}. 
 A variety of interesting approaches have been proposed for the analysis and design of
NCSs. One class of design methods takes advantage of the fact that in many 
communication protocols data is sent in large time-stamped packets. This opens
the possibility to conceive \emph{Sequence-based
  Control} (SBC) formulations, where,  through buffering at the receiver, time delays and packet 
dropouts are to some extent compensated for. Recent works on SBC
for network models with  bounded packet loss and/or bounded time delays include
\cite{tansil07,polliu08,munchr08,liumun09,wanliu10,pinpar11,quenes11a}, which 
adopt worst-case methods for stability analysis. 
As shown in the above works, the (somewhat artificial) assumption that  network
artifacts are bounded is required in order to
be able to establish deterministic stability guarantees of SBC.   Given the widespread adoption of
stochastic models with unbounded support in both the communications 
and also in the systems control community (see,
e.g.,\cite{goldsm05,bergal02,schena08,shi09b,linlem04,guphas07}), it is  
surprising that stochastic models with unbounded support
have only recently been taken into account when studying stability of SBC. In particular,\cite{queost11}
investigates quantized control of  LTI systems with independent and
identically distributed (i.i.d.) dropouts
and establishes sufficient condition for mean-square stability of the NCS,
whereas\cite{quenes12a} focuses on stochastic stability of nonlinear plant models
 in the presence of Markovian dropouts. The work \cite{fisdol13a} studies
 an SBC formulation
 for LTI systems, where the network  introduces possibly unbounded 
dropouts and delays, under the assumption that  acknowledgments of receipt are 
always available.

\begin{figure}[t]
  \centering
  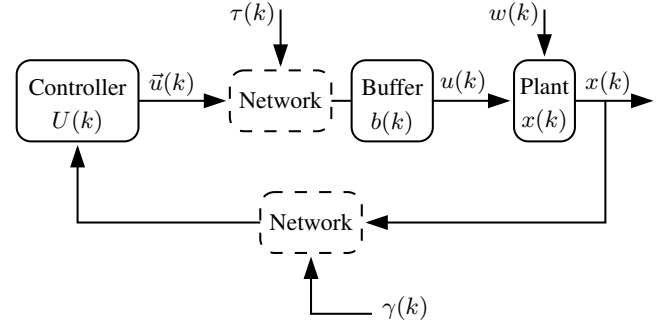
  \caption{NCS architecture with buffer at the actuator node}
  \label{fig:setup_cdc12}
\end{figure}

\par In the present work, we study an SBC formulation for discrete-time
non-linear plant models, where plant states and tentative input sequences are transmitted
over an unreliable communication channel, see Fig.~\ref{fig:setup_cdc12}. As in
our previous works\cite{queost11,quenes12a}, the
control scheme studied
does not require  acknowledgments of receipts or any knowledge of the packet 
delay/dropout distributions.   Instead, control sequences are designed by 
applying a pre-designed nominally stabilizing state-feedback control mapping to
predicted plant 
states. Thus, control values are calculated sequentially, reutilizing the
already computed values. This is computationally attractive and differs from our
previous works  \cite{queost11,quenes12a}, which examine the design of
packetized optimization-based model predictive
controllers. 
Through the development of a Markovian model of the NCS at hand, we show how our recent results
in\cite{quegup11a} can be adapted to the present situation and establish sufficient conditions for
stochastic stability of the NCS class considered. Our current results are applicable to
input-constrained open-loop unstable plants and to networks    which introduce
i.i.d.\ time-delays and i.i.d.\   packet dropouts, in which case the number of
consecutive dropouts have a geometric distribution.

\paragraph*{Notation}
\label{sec:notation}
%\paragraph*{Notation:}
%\label{sec:notation}
We write %$\R$ for the real numbers,  
$\R_{\geq 0}$ for $[0,\infty)$, $\N$ for $\{1,
2, \ldots\}$, and $\N_0$ for $\N \cup \{0\}$. %The operator $\kronecker$ refers to
%the Kronecker product. 
The $p\times p$ identity matrix is
denoted $I_p$. For the all zeroes matrix in $\R^{p\times
  q}$, we write $0_{p\times q}$. The unit
column vector in $\R^{p}$ is $\mathbf{1}_{p}=
[\begin{matrix}
  1&0&\dots&0
\end{matrix}]^T$,
whereas $\mathbf{0}_{p}=0\cdot\mathbf{1}_{p}$. % For sequences, we write 
% $\{y\}_{\K}=\{y(\ell) \;\colon \ell \in \K\}$, and %where $\K\subseteq \N_0$ and 
% \begin{equation*}
%   \{y\}_{\ell_1}^{\ell_2} =
%   \begin{cases}
%     \{y(\ell_1), \dots, y(\ell_2)\}  &\mbox{if $\ell_1 \leq \ell_2,$}\\
%     \{\,\}  &\mbox{if $\ell_1 > \ell_2.$}
%   \end{cases}
% \end{equation*}
%We adopt the
%convention 
%$\sum_{k=\ell_1}^{\ell_2}a_k = 0$,  if $\ell_1 > \ell_2$ and irrespective of
%$a_k$.
%$\{\nu (k)\}_0^k$, or
%
%$\{y\}_{\ell_1}^{\ell_2} = \{y(\ell_1), \dots, y(\ell_2)\}$, for
%$\ell_1 \leq \ell_2$, and $\{y\}_{\ell_1}^{\ell_2} = \{\,\}$, if
%$\ell_1 > \ell_2$.
The norm of a vector $x$ is denoted $|x|$.  A
function $\varphi\colon \R_{\geq 0}\to \R_{\geq 0}$ is of
\emph{class-}$\mathscr{K}_\infty$ ($\varphi \in \mathscr{K}_\infty$), if it is
continuous, zero at zero, strictly increasing, and  unbounded.
$\Prob\{\cdot\}$ refers to probability and $\E\{\cdot\}$,  to expectation.

\section{NCS Architecture} 
\label{sec:pack-pred-contr}
% In this section, we describe the NCS architecture of interest, which is depicted in Fig.~\ref{fig:setup_cdc12}.
% \subsubsection{Plant Model}
% \label{sec:plant-model}
% 
We consider   (possibly unstable)  plant models with state
$x\in \R^n$ and constrained input
$u\in \U\subseteq \R^p$, with $\mathbf{0}_p\in\U$. The plant state trajectory obeys
the recursion:
\begin{equation}
  \label{eq:34}
  x(k+1) = f(x(k),u(k)),\quad k\in\N_0,
\end{equation}
where $f(\mathbf{0}_n,\mathbf{0}_p)=\mathbf{0}_n$. The initial state $x(0)$ is arbitrarily distributed.
 Throughout this work, we will assume that the nominal plant model~\eqref{eq:34}
is globally stabilizable via state
feedback. We adopt an emulation-based approach (cf.,\cite{heetee10}), for a pre-designed
controller $\kappa$, which is nominally stabilizing   in the absence of network
effects (see Section~\ref{sec:example} for a specific example):
\begin{ass}
  \label{ass:CLF}
  There exist $\kappa \colon \R^n\to\U$,  $V\colon
\R^n\to\R_{\geq 0}$,  $\varphi_1, \varphi_2\in\mathscr{K}_\infty$,  
$\rho\in [0,1)$, 
 such that
\begin{equation}
  \label{eq:3b}
  \begin{split}
    \varphi_1(|x|)\leq V(x)&\leq \varphi_2(|x|),\\
%    V(f(x,\kappa(x),\mathbf{0}_m),0) &\leq \rho_{10} V(x,1)\\
    V\big( f(x,\kappa(x))\big) &\leq \rho V(x),  \quad  \forall x\in\R^n.
  \end{split}
\end{equation}
\end{ass}

\subsubsection{Network Effects and Buffering}
 \label{sec:network-effects}
We  focus on a situation where transmissions from the sensor to the
actuator are over a wireless channel. Due to fading and
interference from other users, random dropouts will
occur\cite{queahl10,schsin07,queahl12,queahl13a}. For our purposes, their effect is described by
introducing the binary transmission success 
process $\{\gamma(k)\}_{k\in\N_0}$, where (see Fig.~\ref{fig:setup_cdc12})
\begin{equation}
\label{eq:10}
  \gamma(k)=
  \begin{cases}
    1 &\text{if the controller  receives ${x}(k)$  at time $k$,}\\
    0 & \text{otherwise.}
  \end{cases}
\end{equation}
The controller calculates  control commands
$\vec{u}(k)$ only at the times when data arrives, i.e., when
$\gamma(k)=1$. Transmissions from controller to actuator use a shared network. Fading, interference and medium access
effects introduce random time-delays and dropouts in the
controller-actuator link\cite{ramsan11b}.   Here, we introduce
$\{\tau(k)\}_{k\in\N_0}$, where
\begin{equation}
\label{eq:11}
  \tau(k)=
  \begin{cases}
    i &\text{if $\gamma(k)=1$ and $\vec{u}(k)$ is received at time $k+i$}\\
 &\qquad\text{at the actuator node,}\\
    \infty & \text{otherwise.}
  \end{cases}
\end{equation}

\par As foreshadowed in the introduction, we are interested in SBC, where
$\vec{u}(k)$ contains
tentative plant inputs  for a finite, and fixed, number of $N$ current and future time
steps. To make the predictive nature of the controller explicit, we write
\begin{equation}
  \label{eq:4}
  \vec{u}(k) =
  \begin{bmatrix}
    u(k;k)\\u(k+1;k)\\ \vdots\\ u(k+N-1;k)
  \end{bmatrix}\in\U^N \subseteq \R^{pN},
\end{equation}
where the values $u(k+\ell;k)$ are \emph{tentative controls} for time $k+\ell$,
calculated at time $k$.
At the  actuator side, the received packets are buffered, see
Fig.~\ref{fig:setup_cdc12}. The
buffer uses the time-stamps to only store (parts of) the most recently generated packets received; older
packets are discarded. The most recently generated packet
contained in the buffer at time $k$ is $\vec{u}(T(k))$, where
\begin{equation}
  \label{eq:3}
 T (k) \eq  \max\{\ell\in\N_0 \colon \ell + \tau(\ell)\leq k\}.
\end{equation}
If $T(k)> k-N$, then
 the buffer state at time $k$ contains elements of $\vec{u}(T(k))$ as per
\begin{equation}
  \label{eq:5}
  b(k)=
    \begin{bmatrix}
    u(k;T(k))\\u(k+1;T(k))\\ \vdots\\ u(T(k)+N-1;T(k))
  \end{bmatrix}
\end{equation}
and the plant input is set to $u(k)=u(k;T(k))$. If  $T(k)\leq k-N$,  then the
buffer is empty, $b(k)=\emptyset$, and the plant input is set to zero.\footnote{Alternatively, one could also hold the previous
  control input; cf.,\cite{schena09}.} We will assume that the buffer is
initially empty: 
\begin{equation}
  \label{eq:19}
  b(0)=\emptyset,\quad u(0)=\mathbf{0}_p.
\end{equation}

\subsubsection{Sequence-Based Controller}
\label{sec:contr-sequ-design}
Throughout this work, we will focus on  the  
challenging \emph{UDP-like} case where  the
network provides no  acknowledgments of receipt. Therefore, 
the buffer contents are not available for
control calculations. 
The control sequences 
$\vec{u}(k)$ in~(\ref{eq:4})  are obtained by applying the mapping $\kappa$
of~(\ref{eq:3b}) on predicted nominal plant states, say $x(k+j;k)$, see~(\ref{eq:34}): %\footnote{The control
%  sequences in~(\ref{eq:31}) are \emph{prediction 
% consistent} in the sense of\cite{fingru11}, see also \cite{polliu08,quesil08}.}
\begin{equation}
  \label{eq:31}
  \begin{split} 
    u(k+j;k) &= \kappa (x(k+j;k)), \quad j\in\{0,1,\dots,N-1\},\\
    x(k+j+1;k)&=
    f(x(k+j;k),u(k+j;k)),\quad x(k;k)=x(k).
  \end{split}
\end{equation}
  With SBC, it can be expected that larger horizon lengths $N$ will give better performance, since more
delay and dropout scenarios can be compensated for. This will become apparent
in our subsequent results, in particular, those included in Section~\ref{sec:example}.

\section{NCS Design Model}
\label{sec:ncs-model}
Due to the presence of time delays and the buffering procedure adopted, future plant
states depend not only upon the current plant  
state $x(k)$, but also on delay realizations and  control packets transmitted
after $T(k)\leq 
k$. To model the situation, we  introduce the overall  NCS state
 $\theta(k)\eq
  \big[\begin{matrix}
    x(k)^T&
    U(k)^T
  \end{matrix}\big]^T$,
where %\footnote{For simplicity, we choose $\theta$ to have fixed dimension,
%  stemming from the worst case,  $T(k-1)< k-\bar N$.}
\begin{equation}
  \label{eq:12}
  \begin{split}
    U(k)&\eq
    \begin{bmatrix}
      C_1 \vec{u}(k-1)\\C_2 \vec{u}(k-2)\\ \vdots\\ C_{\bar{N}-1}
      \vec{u}(k-\bar{N}+1)
    \end{bmatrix}\! \in \R^{\nu p}, \\ C_i&\eq
    \begin{bmatrix}
      0_{( N-i)p \times ip} & I_{( N-i)p}
    \end{bmatrix}\!, \\
    \nu&\eq(\bar N-1)\bar N /2,
  \end{split}
\end{equation}
% In~(\ref{eq:12}),
% \begin{equation*}
%   \begin{split}
%  C_i&\eq
%     \begin{bmatrix}
%       0_{( N-i)p \times ip} & I_{( N-i)p}
%     \end{bmatrix}\\
%     \nu&\eq(\bar N-1)\bar N /2\\
%     \bar{N}&\eq \max(\Gamma \cap \{0,1,\dots,N\})
%     %\begin{bmatrix}
%    %   \mathbf{0}_{i} \\ \mathbf{1}_{N-i}
%    %\end{bmatrix}^T \kronecker I_p.
%   \end{split}
% \end{equation*}
with 
     $\bar{N} \eq \max\{\Gamma \cap \{0,1,\dots,N\}\}$ and
where
$   \Gamma \eq \{i\in \N_0  \colon \Prob\{\tau(k)=i\}>0\}$. 
\par  Our subsequent analysis makes use  of the \emph{buffer
  length} process $\{\lambda(k)\}_{k\in\N_0}$, defined via, 
\begin{equation}
  \label{eq:6}
  \lambda(k)\eq
  \begin{cases}
    \max(0,T(k)+N-k) &\text{if $k\in \N$},\\
    0 &\text{if $k =0$},
  \end{cases}
\end{equation}
where $T(k)$ is as in~(\ref{eq:3}).  In view of~(\ref{eq:5}), if $\lambda(k)>0$,
then we have $b(k)\in\U^{\lambda(k)}$, whereas $\lambda(k)=0$ refers to an empty buffer,
$b(k)=\emptyset$. The  plant inputs
are  given by:
\begin{equation}
  \label{eq:7}
  u(k)=
  \begin{cases}
    u(k;\lambda(k)+k-N)&\text{if $\lambda(k)>0$,}\\
    \mathbf{0}_p& \text{if $\lambda(k)=0$.}
  \end{cases}
\end{equation}
 Note that $\lambda(k)$ also serves to describe the age of the buffer content;
 cf.,\cite{pinpar11}.

\par Expression~(\ref{eq:7}) now gives that, if $\lambda(k)>0$, then
$T(k)=\lambda(k)+k-N$, thus
\begin{equation*}
    u(k)=u(k;\lambda(k)+k-N)=G_{\lambda(k)}U(k) + E_{\lambda(k)} \vec{u}(k)
\end{equation*}
where
\begin{equation*}
  \begin{split}
    G_{\lambda(k)}&=\begin{cases}
          0_{p\times \nu p}&\text{if $\lambda(k)=N$,}\\
      \begin{bmatrix}
        0_{p\times   \eta p} &I_{p} &0_{p\times (\nu -1-\eta) p}
      \end{bmatrix}&\text{if $\lambda(k)<N$,}
    \end{cases}\\
    E_{\lambda(k)}&=
    \begin{cases}
      \begin{bmatrix}
        I_{p} &0_{p\times (N-1)p}
      \end{bmatrix}&\text{if $\lambda(k)=N$,}\\
      0_{p\times Np}&\text{if $\lambda(k)<N$,}
    \end{cases}\\
    \eta&=N-\lambda(k)+(N-\lambda(k)-1)\bar N \\
    &- {(N-\lambda(k))(N-\lambda(k)-1)}/{2} .
  \end{split}
\end{equation*}
On the other hand, irrespective of $\lambda(k)$, we have
\begin{equation*}
  \begin{split}
    U(k+1) &= S U(k) +
    \begin{bmatrix}
      I_p\\ 0_{(\nu -1)p\times p}
    \end{bmatrix}
    \vec{u}(k),\\
    S&\eq \begin{bmatrix} 0_{p\times \nu p}\\ \begin{bmatrix}
        I_{(\nu-1)p} & {0}_{(\nu-1)p \times p}
      \end{bmatrix}
    \end{bmatrix}.
  \end{split}
\end{equation*}
The above expressions lead to the jump non-linear NCS model:
\begin{equation}
\label{eq:16}
\begin{split}
    \theta(k+1) &= F_{\lambda(k)}(\theta(k),\vec{u}(k)),\quad k\in\N_0,\\
     F_{\lambda}(\theta,\vec{u}) &\eq
    \begin{bmatrix}
      f(x,G_{\lambda}U + E_{\lambda} \vec{u})\\
      S U +
      \begin{bmatrix}
        I_p\\ 0_{(\nu p-p)\times p}
      \end{bmatrix}
      \vec{u}
    \end{bmatrix},\quad \theta=
    \begin{bmatrix}
      x\\ U
    \end{bmatrix}.
  \end{split}
\end{equation}
In the following section, we will study  a particular instance of the NCS
model obtained. Section~\ref{sec:stochastic-stability} then incorporates the control 
sequences $\vec{u}(k)$, designed as in~\eqref{eq:31},  to study  
stability of the NCS. 

\section{The i.i.d.\ Case}
\label{sec:i.i.d.-case}
In the sequel, we  focus on networks, where 
delays and dropouts are uncorrelated and make the following two assumptions.
\begin{ass}
  \label{ass:iid}
  The transmission success process $\{\gamma(k)\}_{k\in\N_0}$ is i.i.d., with
  $$ \Prob \{\gamma(k)=1\}=q.$$
   The process $\{\tau(k)\}_{k\in\N_0}$ is conditionally i.i.d., with delay distribution
 $$\Prob \{\tau(k)=i\,|\, \gamma(k)=1\}=p_i,\quad i\in\N_0,$$ 
and dropout probability $\Prob \{\tau(k)=\infty\,|\, \gamma(k)=1\}=p_\infty$.\hfs
\end{ass}
A direct consequence of our model is that the unconditional distribution
of $\{\tau(k)\}_{k\in\N_0}$ satisfies, for all $i\in\N_0$,
\begin{equation}
  \label{eq:13}
  \begin{split}
     \Prob \{\tau(k)=i\} &= \Prob\{\gamma(k)=1\} \Prob \{\tau(k)=i\,|\, \gamma(k)=1\}
      = q p_i,\\
     \Prob \{\tau(k)=\infty\}&= q p_\infty
     +  \Prob\{\gamma(k)=0\}\Prob \{\tau(k)=\infty\,|\, \gamma(k)=0\}\\
     &\qquad =  q\cdot p_\infty +  (1-q)
  \end{split}
\end{equation}
 Note that, if $p_\infty>0$ or $q<1$, then the number of consecutive packet
 dropouts in the loop has
unbounded support. Thus, even though the controller transmits  sequences,
 with non-zero probability the plant will be left in open loop during
 intervals of infinite length. This begs the question of  stability of
 this NCS architecture, which we address in Section~\ref{sec:stochastic-stability}.

\begin{rem}[Special cases]
  A particular case of our model results when $q=1$ and $p_i=0$ for all $i\in \N$, in
which case  the controller to actuator network reduces to an erasure channel with dropout rate
$p_\infty=1-p_0$. The latter model was adopted for examining SBC over bit-rate limited
channels in\cite{queost11}.  Another instance studied in previous literature is where $q=1$ and time-delays have a
bounded support, i.e., there exists $\tau^{\text{max}}\in \N_0$,
such that
\begin{equation}
  \label{eq:1}
  p_i = 0,\quad \forall i>\tau^{\text{max}},\; i\not=\infty,
\end{equation}
in which case the dropout rate is given by $$p_\infty = 1
-\sum_{i=0}^{\tau^{\text{max}}}p_i.$$ For the special case of~(\ref{eq:1}) with
no dropouts, we have $\sum_{i=0}^{\tau^{\text{max}}}p_i=1$ as considered, for
example, in\cite{liumu06,tansil07,quenes11a}. \hfs 
\end{rem}

  A key property of the NCS model obtained is that, for    delays and dropouts
  satisfying Assumption~\ref{ass:iid}, the process 
$\{\lambda(k)\}_{k\in\N_0}$ introduced in~(\ref{eq:6}) 
constitutes a homogeneous (finite) Markov Chain. Its transition probabilities
%\begin{equation*}
$$  p_{ij}\eq\Prob\{\lambda(k+1)=j \,|\,\lambda(k)=i \} $$
%\end{equation*}
%and associated transition  matrix 
%\begin{equation}
%  \label{eq:24}
%  \mathcal{P}=[p_{ij}], \quad i,j, \in \{0,1,\dots,N\}
%\end{equation}
are as follows:\footnote{A related characterization can also be found in
  parallel work documented in\cite{fisdol13a}.}
\begin{lem}
\label{lem:Markov}
 If Assumption~\ref{ass:iid} holds, then the transition probabilities
 of $\{\lambda(k)\}_{k\in\N_0}$  satisfy: 
  \begin{equation}
    \label{eq:21}
    p_{ij} =
      \begin{cases}
        q\cdot p_{N-j}  &\text{if $1\leq i\leq j$}\\
        1-q\sum_{\ell=0}^{N-i} p_\ell &\text{if $i= j+1$}\\
        0 &\text{if $i\geq j+2$}\\
        p_{1j} &\text{if $i=0$,}
      \end{cases}
\end{equation}
where $ i,j\in\{0,1,\dots,N\}$.
\end{lem}
\begin{proof}
See Appendix~\ref{proof:markov}.
\end{proof}

\begin{figure}[t]
  \centering
  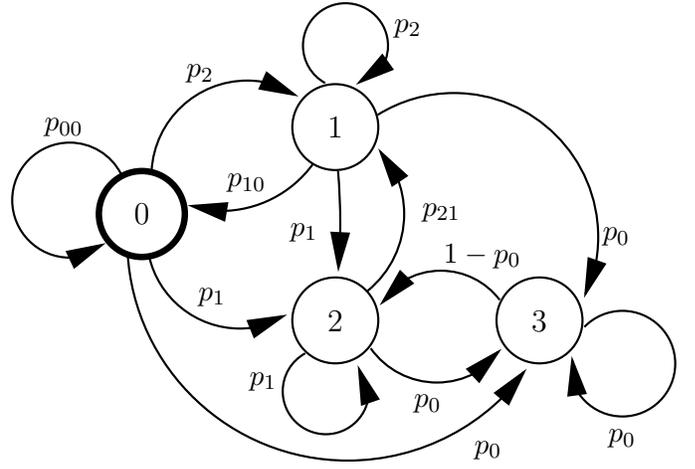
  \caption{Transition graph of $\{\lambda(k)\}_{k\in\N}$ in~(\ref{eq:22}). Here, 
    $p_{10}=p_{00}=1-p_0-p_1-p_2$ and $p_{21}= 1-p_0-p_1$.}
  \label{fig:transitions}
\end{figure}

\begin{ex}
   With horizon length $N=3$ and $q=1$, expression~(\ref{eq:21}) yields the transition matrix
  \begin{equation}
    \label{eq:22}
    [p_{ij}]=
    \begin{bmatrix}
      1-p_0-p_1-p_2 & p_2 & p_1 & p_0\\
      1-p_0-p_1-p_2 & p_2 & p_1 & p_0\\
      0 & 1-p_0-p_1 &  p_1 & p_0\\
       0 &  0 & 1-p_0 & p_0
    \end{bmatrix}.
  \end{equation}
  Fig.~\ref{fig:transitions} illustrates the associated transition graph.\hfs
\end{ex}

% \begin{rem}
%   As a consequence of Lemma~\ref{lem:Markov},  if the
%   plant~(\ref{eq:34}) is LTI, then the NCS model~(\ref{eq:16}) becomes a Markov
%   jump-linear system with jump process $\{\lambda(k)\}_{k\in\N_0}$; %. Thus, SBC
%   %for LTI plants with i.i.d.\ delays and dropouts can, in principle, be designed
%   %using some of the methods described in
%   %\cite{cosfra05,valbas99,cosval00,parkwo02,varval04}.  \
% see; e.g.,\cite{cosfra05}.\hfs % Our MTNS2012 submission minimizes the expected
%   % value, given $x(k)$, of a quadratic cost function with horizon $N$, subject
%   % to the constraint
%   % \begin{equation}
%   %   \lambda(k+\ell)\leq N-\ell.
%   % \end{equation}
% \end{rem}

We will next adapt
 ideas, which were used
in\cite{quegup11a} for the analysis of  control loops without
time-delays or dropouts, but where
processor availability is random. 
Our analysis differentiates between time steps where $\lambda(k)=0$ and, thus, the
buffer is empty, and those time instances where the buffer contains plant
inputs which were calculated by the controller. For future reference, we denote the
 times when $b(k)=\emptyset$  via
$\K=\{k_i\}_{i\in\N_0}$, where 
$k_{0}=0$, see~(\ref{eq:19}), and
\begin{equation}
\label{eq:8}
  k_{i+1} = \inf \big\{ k\in\N \colon k>k_i,\;  \lambda(k)=0\big\},\quad i\in\N_0.
\end{equation}
We also introduce the  process
$\{\Delta_i\}_{i\in\N_0}$, where 
$$  \Delta_i\eq k_{i+1}-k_i,  \quad \forall (k_{i+1}, k_i) \in\K \times\K$$  are the number of time steps between  consecutive
instants in $\K$. This process thereby
amounts to  the
first return times of    $\{\lambda(k)\}_{k\in\N_0}$ to state $0$ and is
i.i.d.\cite{kemsne60}; see also Fig.~\ref{fig:transitions}. The distribution of
$\{\Delta_i\}_{i\in\N_0}$ can be
characterised as follows:
\begin{lem}
  \label{lem:return}
  Suppose that Assumption~\ref{ass:iid} holds and define
\begin{equation}
  \label{eq:28}
    \sigma^T\eq q\begin{bmatrix}
         p_{N-1} & \dots &p_1 &p_0
       \end{bmatrix},\;
    {\mathcal{P}}\eq
    \begin{bmatrix}
      p_{11} & p_{12}& \dots &p_{1N}\\
      \vdots &\vdots & &\vdots \\
      p_{N1} & p_{N2}& \dots &p_{NN}      
    \end{bmatrix},
\end{equation}
with $p_{ij}$ as in~(\ref{eq:21}). Then
  \begin{equation}
    \label{eq:18}
    \Prob\{\Delta_i=j\}
     =
     \begin{cases}
        1-q\sum_{\ell=0}^{N-1} p_\ell&\text{if $j=1$,}\\
      \big(1-q\sum_{\ell=0}^{N-1} p_\ell\big) \sigma^T
    {\mathcal{P}}^{j-2}
       \mathbf{1}_N
       &\text{if $j\geq 2$.}
     \end{cases}
   \end{equation}
\end{lem}
\begin{proof}
  See Appendix~\ref{proof:return}.
\end{proof}
It is intuitively clear that, to achieve good control performance,
$\Delta_i$ should be ``large''. This observation will be confirmed by the 
results established in the following section.

% \begin{ex}
%   For $N=3$ and $q=1$, using~(\ref{eq:22}), Lemma~\ref{lem:return} provides
%   \begin{equation*}
%     \begin{split}
%       \Prob\{\Delta_i=1\}&=1-p_0-p_1-p_2\\
%       \Prob\{\Delta_i=j\}&=\begin{bmatrix}
%          p_{2} &p_1 &p_0
%        \end{bmatrix}
%         \begin{bmatrix}
%        p_2 & p_1 & p_0\\
%       1-p_0-p_1 &  p_1 & p_0\\
%         0 & 1-p_0 & p_0
%     \end{bmatrix}^{j-2}
%     \begin{bmatrix}
%       1-p_0-p_1-p_2\\
%       0\\
%       0
%     \end{bmatrix}\!,\quad \forall j\geq 2.
%   \end{split}
% \end{equation*}
% \end{ex}

\section{Stability Analysis}
\label{sec:stochastic-stability}
 Due to the occurrence of random time-delays and dropouts, the plant input is
 also random. In particular, at all  
instants where  $\lambda(k)<N$, the plant is unavoidably left in open-loop; if
$\lambda(k)=0$, then 
$u(k)=\mathbf{0}_p$.   This observation motivates us to study stochastic
stability of the NCS model. 

%\par Various stochastic stability notions  have been studied; see, e.g.,~\cite{jichi91,kushne71}.% We here focus on 
\begin{defi}[Stochastic Stability]
  The NCS is stochastically
  stable, if for some $\varphi \in\mathscr{K}_\infty$, the plant state
  trajectory $\{x(k)\}_{k\in\N_0}$ satisfies
  $\sum_{k={0}}^{\infty}\E\big\{\varphi(|x(k)|)\big\} <\infty$.\hfs
\end{defi}

% \begin{rem}
% Stochastic stability implies \emph{asymptotic stability}, i.e.,
% $\lim_{k\to\infty} \E \big\{\varphi(|x(k)|)\big\} =0$. For $\varphi(s)=s^2$, the
% latter amounts to
% \emph{mean square stability}, 
% $\lim_{k\to\infty} \E \{|x(k)|^2\} =0$, see~\cite{jichi91}. \hfs
% \end{rem}

 Assumption~\ref{ass:bound_prob}  bounds the
rate of increase  
of  $V$ in~(\ref{eq:3b}), when the  plant input is zero. It also
imposes a restriction on 
  the distribution of the  initial plant state, see Section~\ref{sec:example}
  for an example.
\begin{ass}
  \label{ass:bound_prob}
  There exists $\alpha\in\R_{\geq 0}$
such that
  \begin{equation}
    \label{eq:14}
     V({f}(x,\mathbf{0}_p))\leq\alpha V(x),\quad\forall x \in\R^n,
   \end{equation}
   and
   $\E\big\{\varphi_2(|x(0)|)\big\}<\infty$, where $V$ and
$\varphi_2$ are as  in~(\ref{eq:3b}).\footnote{If $\alpha$   cannot be
  found for given $V$, then the current results are non-informative and the
  question of controller-redesign 
  arises.} \hfs
\end{ass}
 The results of Section~\ref{sec:i.i.d.-case} allow one to adopt a stochastic Lyapunov
function approach to study   stability of the NCS described
by~(\ref{eq:34}),~(\ref{eq:31}) and~(\ref{eq:7}). We begin by stating the
following lemma:
\begin{lem}
\label{lem:drift}
  Suppose that Assumptions~\ref{ass:CLF} to~\ref{ass:bound_prob} hold and
  consider $k_0,k_1\in\K$, see~\eqref{eq:8}. Then 
  \begin{equation*}
    \begin{split}
      \E&\big\{V(x({k_{1}}))\,\big|\,\theta(k_{0})=\vartheta\big\}\leq \alpha
      \sum_{j\in\N} \Prob\{ \Delta_i =j\} \rho^{j-1} V(\chi),\\
      &\forall
      \vartheta=
      \begin{bmatrix}
        \chi\\ \mathcal{U}
      \end{bmatrix}
      \in \R^{n+\nu p},
    \end{split}
\end{equation*}
where the distribution of $\{\Delta_i\}_{i\in\N_0}$ is given in~(\ref{eq:18}).
\end{lem}
\begin{proof}
See Appendix~\ref{proof:drift}.
\end{proof}

\begin{figure}[t] 
  \includegraphics[width=\columnwidth]{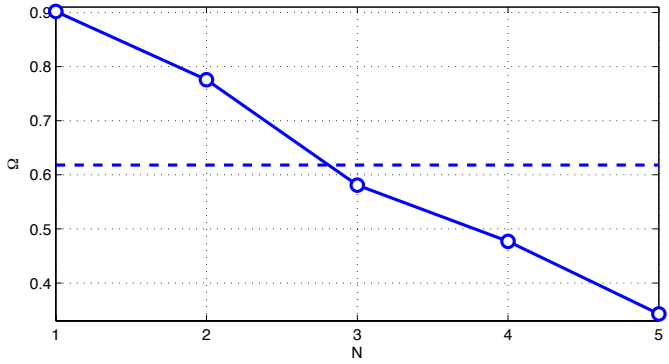} 
\caption{$\Omega$ as a function of  $N$. The horizontal line
      depicts the stability condition, $\Omega\leq
      1/\alpha \approx 0.618$.}
    \label{fig:OmegaN}
  \end{figure} 

Whilst, in general,
  $\{\theta(k)\}_{k\in \N_0}$ is not Markovian, $\{\theta(k_i)\}_{k_i\in\K}$ is
  Markovian and we can use Lemma~\ref{lem:drift} to
derive the following sufficient condition for stochastic stability of the NCS:
\begin{thm}
  \label{theorem:a1_stability}
  Suppose that Assumptions~\ref{ass:CLF} to~\ref{ass:bound_prob} hold and define 
\begin{equation}
  \label{eq:25}
  \Omega\eq  \bigg(1-q\sum_{\ell=0}^{N-1} p_\ell\bigg) \bigg(1+ 
 \rho \sigma^T \sum_{j\in\N_0}
       \big(\rho\mathcal{P}\big)^{j} \mathbf{1}_N\bigg)<\infty,
\end{equation}
see~(\ref{eq:28}).
If $\Omega<1/\alpha$, then the  NCS described
by~(\ref{eq:34}),~(\ref{eq:31}) and~(\ref{eq:7})  is stochastically stable.
\end{thm}
\begin{proof}
See Appendix~\ref{proof:thm}.
\end{proof}
The above result establishes  conditions on the network, plant and
controller, which guarantee that the closed loop in the presence of
 dropouts and delays will be stochastically stable.  It is worth noting that with
   SBC, control packets which are  delayed longer than $N-1$ time-steps are
 discarded. This is reflected in~(\ref{eq:25})  by the fact that $\Omega$
 depends on $p_0,\dots,p_{N-1}$, 
 $q$ and $\rho$ only. We also emphasize that the condition $\Omega \alpha<1$
 does not require contractiveness for 
  each step $k\in\N$. Instead, it only amounts contractiveness of the process  sampled
 at 
  $k\in\K$.

 \begin{rem}[Disturbances]
  The present work focuses on controlling a disturbance-free system
  model of the form~(\ref{eq:34}). Our current results can be extended to
  encompass non-zero 
disturbances, provided that continuity assumptions are imposed on the plant
model and the feedback law $\kappa$. Such an analysis was carried out for a
related problem in our
recent paper\cite{quenes12a}. By constraining sample-path
differences between the nominal and the perturbed model to be close in a sample
path sense, \cite{quenes12a} established the
boundedness of moments for the plant state process. Such an approach can be adopted
for the present control  formulation as well. \hfs
 \end{rem}

\begin{figure}[t] 
\includegraphics[width=\columnwidth]{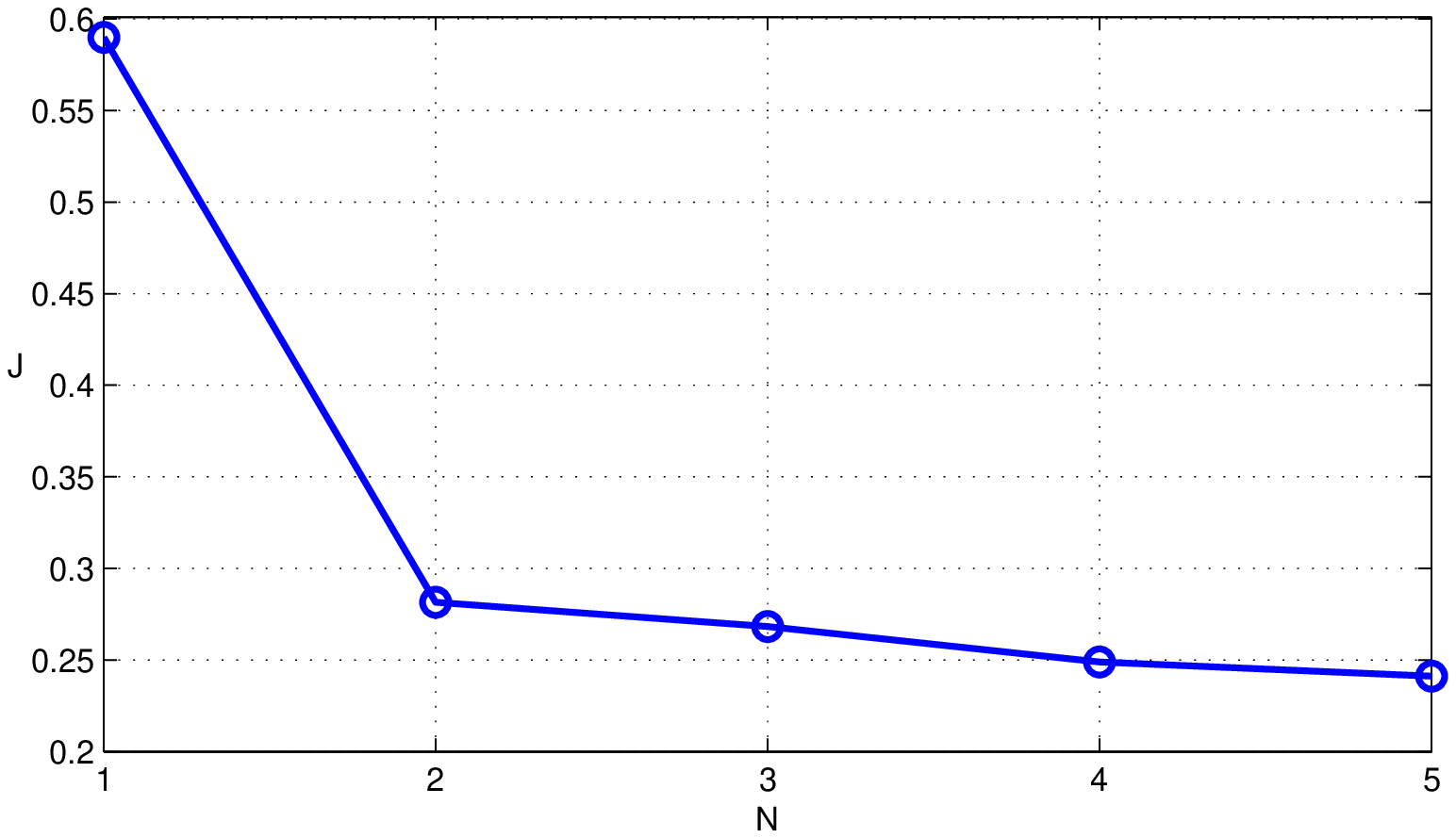} 
\caption{Empirical expectation of $J\eq(1/50)\sum_{k=0}^{49} |x(k)|^2$,   as a
  function of the  buffer length $N$.}  
\label{fig:J2}
\end{figure}

\section{Simulation Study}
\label{sec:example}
Consider 
an open-loop unstable constrained plant model of the form~\eqref{eq:34}, where
\begin{equation}
  \label{eq:78}
  f(x,u) =
  \begin{bmatrix}
    x_2+u_1\\-{\rm{sat}} (x_1+x_2) +u_2
  \end{bmatrix},
\end{equation}
with
\begin{equation*}
x=
\begin{bmatrix}
  x_1\\x_2
\end{bmatrix},\;
u=
\begin{bmatrix}
  u_1\\u_2
\end{bmatrix},\quad
  \rm{sat}(\mu)=
  \begin{cases}
    -1,&\text{if $\mu<-1$,}\\
    \mu & \text{if $\mu \in [-1,1]$,}\\
    1,&\text{if $\mu>1$},
  \end{cases}
\end{equation*}
see\cite[Example 2.3]{khalil96} and\cite{quenes12a}. The initial state is taken from a Gaussian
distribution with zero mean and variance $I_2$. We fix $|\cdot|$ as the Euclidean norm. 
 The second component of the  plant input is constrained via
$|u_2(k)|\leq 0.8$, $\forall k\in\N_0$, thus, $\U=\R\times [-0.8,0.8]$. 
\par The    model~\eqref{eq:78} can be globally stabilized (in the absence of network effects) by the constrained  control law $\kappa\colon
\R^2\to\U$, where 
$$\kappa(x)=\big[\begin{matrix}
   - x_2& 0.8 {\rm{sat}} (x_1+x_2)
  \end{matrix}\big]^T.$$
In fact, with $V(x)=2|x|$,  direct calculations give
\begin{equation*}
  \begin{split}
    V\big(f(x&,\kappa(x)) \big) =
  0.4 | {\rm{sat}} (x_1+x_2)| \leq  0.4 | x_1+x_2|\\
  &\leq 0.8\max\{|x_1|,|x_2|\}-\max\{|x_1|,|x_2|\}+|x|\leq |x|.
  \end{split}
\end{equation*}
% \begin{equation*}
%  \begin{split}
%    V&\big(f(x,\kappa(x),\mathbf{0}_s) \big) =
%   0.4 | {\rm{sat}} (x_1+x_2)| \leq  0.4 | x_1+x_2|\\
%   &\leq 0.8\max\{|x_1|,|x_2|\}-\max\{|x_1|,|x_2|\}+|x|\leq |x|.
% \end{split}
% \end{equation*}
Thus, Assumption~\ref{ass:CLF} holds with $\rho = 1/2$, and
$\varphi_1(s)=\varphi_2(s)=2s$. Furthermore, by proceeding as
in\cite[p.73]{khalil96}, it can be shown that~(\ref{eq:14})
holds with %open-loop rate of growth bound constant
$\alpha =1.618$.

The network introduces i.i.d.\ delays and dropouts satisfying
Assumption~\ref{ass:iid}, with $q=0.9$, $p_0=0.2$, $p_1=p_2=0.25$,
$p_3=p_4=0.1$, $p_5=p_\infty=0.05$.
Fig.~\ref{fig:OmegaN} depicts the associated quantity $\Omega$,
see~\eqref{eq:25}, as a function of the horizon length (buffer size) $N$. By recalling 
Theorem~\ref{theorem:a1_stability}, the figure illustrates that using larger
horizons is beneficial for guaranteeing stability. In particular, the NCS is
stochastically  stable for horizon lengths
$N\geq 3$.

% \begin{figure}[t]
% \centering
% \subfloat[][\label{fig:OmegaN}$\Omega$ as a function of  $N$. The horizontal line
%       depicts the sufficient condition for  stability, $\Omega\leq
%       1/\alpha\approx 0.618$.]{
% \includegraphics[width=.45\columnwidth]{OmegaN2}
% }
% \subfloat[][\label{fig:X2}Trajectory of $|x|^2$ for different cases.]{
% \includegraphics[width=.45\columnwidth]{X2}
% }
% \end{figure}

% \begin{figure}[t]
% \centering
% \includegraphics[width=.45\columnwidth]{OmegaN2}
%       \caption{$\Omega$ as a function of  $N$. The horizontal line
%       depicts the sufficient condition for stochastic stability, $\Omega\leq
%       1/\alpha\approx 0.618$.}
%       \label{fig:OmegaN}
% \end{figure}

% \begin{figure}[t]
% \centering
% \includegraphics[width=.45\columnwidth]{X2}
%       \caption{Trajectory of $|x|^2$ for different cases.}
%       \label{fig:X2}
% \end{figure}
% \par We now also include disturbances, 
% \begin{equation}
%   \label{eq:78b}
%   f(x,u,w) =
%   \begin{bmatrix}
%     x_2+u_1\\-{\rm{sat}} (x_1+x_2) +u_2
%   \end{bmatrix}+
% \begin{bmatrix}
%  \sqrt{w^2+5} -\sqrt{5}\\
%   0
% \end{bmatrix},
% \end{equation}

Fig.~\ref{fig:J2} compares the empirical expectation of $|x(k)|^2$ as a
function of $N$; results are averaged out
over 100 realizations of the initial state and transmission outcome realizations. It can be seen that performance improves monotonically
with the buffer length.

\section{Conclusions}
\label{sec:conclusions} 
We have studied a  control formulation for  non-linear constrained
plant models where communications between controller and plant are affected by
random dropouts and delays. To compensate for  transmission effects,
the controller transmits sequences of tentative
plant inputs to a buffer   at the actuator node. By developing a NCS model
which makes the size of the buffer contents explicit, we have established
sufficient conditions for stochastic stability when dropouts and delays are
i.i.d. Future work of interest includes studying temporally and spatially
correlated 
time-delays and dropouts. For that purpose, we foresee that the models developed
in\cite{queahl12,queahl13a} 
could serve as a starting point.

\bibliography{/Users/daniel/Dropbox/dquevedo}
\appendices

\appendix

\subsection{Proof of Lemma~\ref{lem:Markov}}
\label{proof:markov}
 We first note that transitions corresponding  to $i \geq j+2$ will never occur,
since the buffer can only loose, at most,  one control value in one transition, thus,
$p_{ij} = 0$, for $i \geq j+2$.   

\par For the cases $1\leq i\leq j$, the buffer is never empty and  some control
sequence generated after $T(k)$ is received at  instant $k+1$.  
Since $\lambda(k)=N-\ell>0$ if and only if $\tau(k-\ell)\leq \ell$
and $\tau(k-\ell+i)>\ell-i$ for all $i\in\{1,2,\dots,\ell\}$, a careful
case-by-case analysis considering all the 
possibilities and taking into account that the size of the buffer is $N$ reveals that 
 \begin{equation*}
    \begin{split}
      &p_{ij}=\Prob\{\lambda(k+1)=j\,|\,\lambda(k)=i\}\\
      &=\Prob\big\{\tau({j-N+k+1})\leq
      N-j \\
      &\qquad \big|\,\tau({j-N+k+1})> N-j-1\big\}\\
      &\quad\times\prod_{\ell =
        2}^{N+1-j}\Prob\Big\{\tau(j-N+k+\ell)> N-j-\ell
      +1\\ 
      &\qquad\qquad\big| \,\tau(j-N+k+\ell)> N-j-\ell\Big\}\\
% &= \frac{\Prob\{\tau(j-N+k+1)\leq N-j\land
%         \tau(j-N+k+1)>N-j-1\}}{\Prob\{\tau(j-N+k+1)>N-j-1\}}\\
%       &\cdot\prod_{\ell
%         = 2}^{N+1-j}\frac{\Prob\{\tau(j-N+k+\ell)>N-j-\ell +1\land
%         \tau({j-N+k+\ell})>N-j-\ell\}}{\Prob\{\tau({j-N+k+\ell})>N-j-\ell\}}\\
        &=\frac{\Prob\{\tau({j-N+k+1})=N-j\}}{\Prob\{\tau({j-N+k+1})>N-j-1\}}
\\ &\quad\prod_{\ell = 2}^{N+1-j}\!\frac{\Prob\{\tau({j-N+k+\ell})>N-j-\ell
        +1\}}{\Prob\{\tau({j-N+k+\ell})>N-j-\ell\}}=q\cdot p_{N-j} 
    \end{split}
  \end{equation*}
since $\{\tau(k)\}_{k\in \N_0}$ is i.i.d.
\par The case $j=i-1$ refers to the event that no ``fresh'' control sequence arrives to the
buffer at the   
instant $k+1$. The associated transition probabilities satisfy: 
  \begin{equation*}
    \begin{split}  
     p_{ij}&=\Prob\{\lambda(k+1)=j\,|\,\lambda(k)=i\}\\
     &= \Prob\big\{\tau({k-N+1+j})\leq
     -N+2+j\\
  &\qquad\qquad \big|\,\tau({k-N+1+j})\leq -N+1+j\big\}\\
     &\quad\times\prod_{\ell =
       0}^{N-1-j}\Prob\{\tau({k+1-\ell})> \ell \,|\,\tau({k+1-\ell})>
     \ell-1\}\\
      &=\prod_{\ell = 0}^{N-1-j}\frac{\Prob\{\tau({k+1-\ell})>
        \ell\land \tau({k+1-\ell})> \ell-1\}}{\Prob\{\tau({k+1-\ell})>
        \ell-1\}}\\
      &=\prod_{\ell =
      0}^{N-1-j}\frac{\Prob\{\tau({k+1-\ell})>\ell\}}{\Prob\{\tau({k+1-\ell})>\ell
      - 1\}} \\
    &= \prod_{\ell =
      0}^{N-1-j}\frac{1-q\sum_{m=0}^{\ell}p_{m}}{1-q\sum_{m=0}^{\ell-1}p_{m}}
    = 1-q\sum_{\ell=0}^{N-1-j}p_{\ell} = 1-q\sum_{\ell=0}^{N-i}p_{\ell}.
    \end{split}
  \end{equation*}
  Finally, for the case $i = 0$, we  consider two different events, namely, $j =
  0$ and $j > 0$. Direct calculations give that  
  \begin{equation*}
    \begin{split}  
      p_{00}&=\Prob\{\lambda(k+1)=0\,|\,\lambda(k)=0\}\\
      &=
      \prod_{\ell =0}^{N-1}\Prob\{\tau({k+1-\ell})>\ell\,|\,
      \tau({k+1-\ell})>\ell -
      1\}\\
      %&= \prod_{\ell =0}^{N-1}
      %\frac{\Prob\{\tau(k+1-\ell)>\ell\land
      %    \tau({k+1-\ell})>\ell - 1\}}{\Prob\{\tau({k+1-\ell})>\ell - 1\}}\\
        &=
    \prod_{\ell =
      0}^{N-1}\frac{\Prob\{\tau({k+1-\ell})>\ell\}}{\Prob\{\tau({k+1-\ell})>\ell
      - 1\}} = p_{10}.
    \end{split}
  \end{equation*} 
  On the other hand, for $j > 0$, we obtain 
  \begin{equation*}
    \begin{split}
      p_{0j}&=\Prob\{\lambda(k+1)=j\,|\,\lambda(k)=0\}\\
      &=\frac{\Prob\{\tau({j-N+k+1})=N-j\}}{\Prob\{\tau({j-N+k+1})>N-j-1\}}\\
&\quad \times \prod_{\ell
        = 2}^{N+1-j}\frac{\Prob\{\tau({j-N+k+\ell})>N-j-\ell
        +1\}}{\Prob\{\tau({j-N+k+\ell})>N-j-\ell\}}
      = p_{1j},
    \end{split}
  \end{equation*}
which proves the result.

\subsection{Proof of Lemma~\ref{lem:return}}
\label{proof:return}
The proof is akin to that of Lemma 2 of\cite{quegup11a}. The case $j=1$ is immediate, since
 $$     \Prob\{\Delta_i=1\}=\Prob\{\lambda(k+1)=0\,|\,\lambda(k)=0\}=p_{00},$$
see~(\ref{eq:21}).
For $j\geq 2$, we analyze the state trajectory of the Markov Chain
$\{\lambda(k)\}_{k\in\N_0}$ and proceed
  as follows: For all $m\in\{1,\dots,N\}$,  denote by $\delta_m$  the
  first passage time of the state 
  $m$ to state $0$. Thus, $\delta_m$ are discrete random
  variables, where $\delta_m=j$, if $\{\lambda(k)\}_{k\in\N_0}$ enters the state
  $0$  from $m$ for the 
  first time in exactly $j$ steps. Since, with exception of state $0$, only the state ${1}$
  can reach $0$ in one step, it is easy to see that
  \begin{equation}
    \label{eq:25b}
       \Prob\{\delta_m = 1\}= p_{m0}=
     \begin{cases}
      1-q\sum_{\ell=0}^{N-1} p_\ell &\text{if $m=1$}\\
      0 &\text{if $m>1$,}
    \end{cases}
\end{equation}  
where we have used~\eqref{eq:21}.
For $j\geq 2$, paths from $m\not = 0$ to $0$ go through intermediate states $\ell\not = 0$,
providing the recursions
$$  \Prob\{\delta_m = j\} =\sum_{\ell =1}^{N} p_{m\ell} \Prob\{\delta_\ell =
  j-1\}, \forall m .$$ Therefore,
  \begin{equation*}
    % \begin{split}
    \begin{bmatrix}
      \Prob\{\delta_1 = j\}\\
      \vdots\\
      \Prob\{\delta_{N} = j\}
    \end{bmatrix}
  = {\mathcal{P}}
    \begin{bmatrix}
      \Prob\{\delta_1 = j-1\}\\
      \vdots\\
      \Prob\{\delta_{N} = j-1\}
    \end{bmatrix},\; \quad \forall j\geq 2,
  \end{equation*}
 yielding the explicit formula:
\begin{equation}
  \label{eq:9}
  \begin{bmatrix}
      \Prob\{\delta_1 = j\}\\
      \vdots\\
      \Prob\{\delta_{N} = j\}
    \end{bmatrix}
  = {\mathcal{P}}^{j-1} \begin{bmatrix}
      \Prob\{\delta_1 = 1\}\\
      \vdots\\
      \Prob\{\delta_{N} = 1\}
    \end{bmatrix}
= \bigg(1-q\sum_{\ell=0}^{N-1} p_\ell\bigg){\mathcal{P}}^{j-1}\mathbf{1}_N,
\end{equation}
which by virtue of~(\ref{eq:25b}) holds, not only for $j\geq 2$, but also for $j=1$.
Thus, the distribution of $\{\Delta_i\}$  can be
 obtained from those of $\{\delta_m\}$ 
  by considering the transitions away
   from $0$:
  \begin{equation*}
    \begin{split}
      \Prob\{\Delta_i=j\}&=\sum_{\ell =1}^{N} p_{0\ell} \Prob\{\delta_\ell =
      j-1\}=\sigma^T\begin{bmatrix}
        \Prob\{\delta_1 = j-1\}\\
        \vdots\\
        \Prob\{\delta_{N} = j-1\}
      \end{bmatrix}\\
      &= \bigg(1-q\sum_{\ell=0}^{N-1}
      p_\ell\bigg)\sigma^T{\mathcal{P}}^{j-2}\mathbf{1}_N,
    \end{split}
\end{equation*}
for all $j\geq 2$, where we have used~\eqref{eq:21} and~(\ref{eq:9}).
\hfs

\subsection{Proof of Lemma~\ref{lem:drift}}
\label{proof:drift}
 Assumption~\ref{ass:bound_prob} gives that
  \begin{equation}
    \label{eq:26}
    V(x(k_0+1))\leq \alpha  V(x(k_0)),\quad \forall x(k_0)\in\R^n.
  \end{equation}
On the other hand, if $k \not\in\K$, then $\lambda(k)>0$ and the plant
input stems from 
current or previously calculated control sequences. Since the control input
sequences are constructed 
following~(\ref{eq:31}), in the
disturbance-free case considered, $x(k)=x(k;T(k))$, where $T(k)$ is as
in~(\ref{eq:3}). Therefore,~(\ref{eq:3b})  gives that
$ V(x(k+1))\leq \rho  V(x(k))$, $\forall k \not \in \K$.
Thus,
\begin{equation*}
  \begin{split}
    \E&\big\{V(x({k_{1}}))\,\big|\,\theta(k_{0})=\vartheta,\Delta_0=j\big\} \leq
    \rho^{j -1}\alpha V(\chi),\\
    &\quad\forall \vartheta=
    \begin{bmatrix}
      \chi\\ \mathcal{U}
    \end{bmatrix}
    \in \R^{n+\nu p}.
  \end{split}
\end{equation*}
Since $\{\Delta_i\}_{i\in\N_0}$ is i.i.d., the result follows directly from
the law of total expectation.  \hfs

\subsection{Proof of Theorem~\ref{theorem:a1_stability}}
\label{proof:thm}
We first note that $$\Omega=\sum_{j\in\N} \Prob\{ \Delta_i =j\} \rho^{j
  -1}.$$ Since $\rho\in[0,1)$, we have $\Omega <\infty$.  Furthermore, 
 $\{\theta(k_i)\}_{k_i\in\K}$ is Markovian. Thus,
Lemma~\ref{lem:drift} and ~\cite[Chapter
  8.4.2, Theorem 2]{kushne71} yield exponential stability at instants
  $k_{i}\in\mathcal{K}$, i.e.,  
  \begin{equation*}
    \E\{V(x({k_{i}}))\,|\,\theta(k_{0})=\vartheta\}
    \leq\big(\alpha \Omega\big)^{i}V(\chi),\quad\forall \vartheta=\!
 \begin{bmatrix}
   \chi\\ \mathcal{U}
 \end{bmatrix}
\!\in \R^{n+\nu p}\!.
  \end{equation*}
For the time instants $k\in\N\backslash \K$, i.e., where $\lambda(k)>0$, the
plant inputs stem from previously calculated control sequences. Thus, we can use~(\ref{eq:3b})
and~(\ref{eq:26}) to bound
\begin{equation*}
  \begin{split}
    \E&\Bigg\{\sum_{k=k_{i}}^{k_{i+1}-1}V(x({k}))\,\bigg|\,x(k_{i})=\chi_i,\Delta_i=j\Bigg\}
    \leq \Bigg(1+ \alpha\sum_{l=0}^{j-2} \rho^l\Bigg) V(\chi_i)\\
    &\leq \bigg(1
    +\frac{\alpha}{1-\rho}\bigg)V(\chi_i),
  \end{split}
\end{equation*}
so that, by the law of total expectation,
\begin{equation*}
  \E\Bigg\{\sum_{k=k_{i}}^{k_{i+1}-1}V(x({k}))\,\bigg|\,\theta(k_{i})=\vartheta_i\Bigg\} 
  \leq 
 \frac{1+\alpha-\rho}{1-\rho}V(\chi_i),\quad \forall \vartheta_i=
 \begin{bmatrix}
   \chi_i\\ \mathcal{U}_i
 \end{bmatrix}.
\end{equation*}
The result follows from  taking conditional expectation  $\E\{\, \cdot\,|\,
\theta(k_0)\}$, use  of the Markovian property of
  $\{\theta(k_i)\}_{k_i\in\K}$, and proceeding, \emph{mutatis mutandis}, as in the proof 
of \cite[Thm.1]{quegup11a}.  

\end{document}